\documentclass[a4paper]{amsart}
\usepackage{amsthm}
\usepackage{amssymb}
\usepackage{enumerate}
\usepackage[mathcal]{euscript}
\usepackage{dsfont}
\usepackage{hyperref}
\usepackage{tikz-cd}

\numberwithin{equation}{section}

\theoremstyle{plain}

\newtheorem{lemma}{Lemma}[section]
\newtheorem{theorem}[lemma]{Theorem}
\newtheorem{proposition}[lemma]{Proposition}
\newtheorem{corollary}[lemma]{Corollary}

\theoremstyle{definition}
\newtheorem{construction}[lemma]{Construction}

\newtheorem{example}[lemma]{Example}

\newtheorem{assumption}[lemma]{Assumption}

\theoremstyle{remark} 
\newtheorem{remark}[lemma]{Remark} 
\newtheorem*{claim}{Claim} 

\newcommand{\cat}{\mathsf}
\newcommand{\dbcat}[1]{\cat{D}^{\mathrm{b}}(\operatorname{mod}#1)}

\newcommand{\End}{\operatorname{End}}
\newcommand{\Ext}{\operatorname{Ext}}

\newcommand{\fgc}{(\textrm{Fg})\,}

\newcommand{\GP}{\operatorname{GP}}
\newcommand{\GInj}{\operatorname{GInj}}
\newcommand{\Ginj}{\operatorname{Ginj}}
\newcommand{\oGInj}{\operatorname{\overline{GInj}}}
\newcommand{\oGinj}{\operatorname{\overline{Ginj}}}
\newcommand{\Gproj}{\operatorname{Gproj}}
\newcommand{\GProj}{\operatorname{GProj}}
\newcommand{\uGproj}{\operatorname{\underline{Gproj}}}
\newcommand{\uGProj}{\operatorname{\underline{GProj}}}

\newcommand{\HH}{\operatorname{HH}}
\newcommand{\Hom}{\operatorname{Hom}}

\newcommand{\PHom}{\operatorname{PHom}}

\newcommand{\inj}{\operatorname{inj}}
\newcommand{\Inj}{\operatorname{Inj}}
\newcommand{\injdim}{\operatorname{inj{.}dim}}

\renewcommand{\mod}{\operatorname{mod}}
\newcommand{\Mod}{\operatorname{Mod}}

\newcommand{\proj}{\operatorname{proj}}

\newcommand{\Proj}{\operatorname{Proj}}
\newcommand{\Rad}{\operatorname{Rad}}

\newcommand{\oHom}{\overline{\Hom}}
\newcommand{\uHom}{\underline{\Hom}}
\newcommand{\sing}[1]{\cat{D}_{\mathrm{sg}}(#1)}
\newcommand{\Spec}{\operatorname{Spec}}

\newcommand{\tors}{\gamma}
\newcommand{\Tr}{\operatorname{Tr}}

\newcommand{\comp}{\circ}

\newcommand{\ges}{{\scriptscriptstyle\geqslant}}

\newcommand{\kos}[2]{{#1}/\!\!/{#2}} 

\newcommand{\lotimes}{\otimes^{\mathbf L}}
\newcommand{\op}{\mathrm{op}}

\newcommand{\da}{{\downarrow}}
\newcommand{\lra}{\longrightarrow}

\newcommand{\xra}{\xrightarrow}

\newcommand{\mcV}{\mathcal{V}}

\newcommand{\sfb}{\mathsf b} 
\newcommand{\sfc}{\mathsf c}

\newcommand{\sfC}{\mathsf C}

\newcommand{\sfT}{\mathsf T} 
\newcommand{\sfU}{\mathsf U}

\newcommand{\bfD}{\mathbf D}

\newcommand{\bbZ}{\mathbb Z} 

\newcommand{\bsa}{\boldsymbol{a}} 
\newcommand{\bsb}{\boldsymbol{b}} 
 
\newcommand{\bsp}{\boldsymbol{p}}

\newcommand{\bst}{\boldsymbol{t}}

\newcommand{\fa}{\mathfrak{a}}

\newcommand{\fm}{\mathfrak{m}} 
\newcommand{\fp}{\mathfrak{p}}
\newcommand{\fq}{\mathfrak{q}} 
\newcommand{\fr}{\mathfrak{r}}

\newcommand{\vf}{\varphi}
\newcommand{\gam}{\varGamma}

\def\Si{\Sigma}

\newcommand{\bsmat}{\begin{smallmatrix}}
\newcommand{\esmat}{\end{smallmatrix}}

\title[Local duality for Gorenstein algebras]{Local duality for the singularity category of a \\ finite
  dimensional Gorenstein algebra}

\author[Benson, Iyengar, Krause, and Pevtsova]{Dave Benson, Srikanth
  B. Iyengar, Henning Krause\\  and Julia Pevtsova}

\address{Dave Benson \\ 
Institute of Mathematics\\ 
University of Aberdeen\\ 
King's College\\ 
Aberdeen AB24 3UE\\ 
Scotland U.K.}

\address{Srikanth B. Iyengar\\ 
Department of Mathematics\\
University of Utah\\ 
Salt Lake City, UT 84112\\ 
U.S.A.}

\address{Henning Krause\\ 
Fakult\"at f\"ur Mathematik\\ 
Universit\"at Bielefeld\\ 
33501 Bielefeld\\ 
Germany.}

\address{Julia Pevtsova\\ 
Department of Mathematics\\ 
University of Washington\\ 
Seattle, WA 98195\\ 
U.S.A.}

\begin{document}

\begin{abstract} 
A duality theorem for the singularity category of a finite dimensional Gorenstein algebra is proved. It complements a duality on the category of perfect complexes, discovered by Happel. One of its consequences is an analogue of Serre duality, and the existence of Auslander-Reiten triangles for the $\fp$-local and $\fp$-torsion subcategories of the derived category, for each homogeneous prime ideal $\fp$ arising from the  action of a commutative ring via Hochschild cohomology.
\end{abstract}


\date{\today}

\thanks{SBI was partly supported by NSF grants DMS-1503044 and DMS-1700985, and JP was  partly supported by NSF grants DMS-0953011 and DMS-1501146. The authors are grateful to the American Institute of Mathematics in San Jose, California for supporting this project by their ``Research in Squares"  program.}

\maketitle

\section{Introduction}
This work concerns duality phenomena in various triangulated categories of modules over Gorenstein algebras. By a Gorenstein algebra we mean here an algebra $A$, finite dimensional over a field $k$, with the property that $A$ has finite injective dimension both as a left module and a right module over itself. For such an $A$ the derived Nakayama functor is an equivalence:
\begin{flalign*}
&&\nu\colon \dbcat A\stackrel{\sim}\lra \dbcat{A}, \quad  \text{where} \quad
  X\xra{\ \nu\ } \Hom_k(A,k)\lotimes_A X\,.
  \end{flalign*}
Happel~\cite{Happel:1987} proved that for perfect complexes $X,Y$ there is a  natural isomorphism
\[
\Hom_k(\Hom_{\cat D}(X,Y),k) \cong \Hom_{\cat D}(Y,\nu X)\,.
\]
where  $\cat D:=\dbcat A$. In other words, the Nakayama functor on $\cat D$ restricts to a Serre functor, in the sense of Bondal and Kapranov \cite{Bondal/Kapranov:1990a}, on the full subcategory of perfect complexes. 

In this work we discover that Happel's result is only the tip of an iceberg: It is a special case of a duality  on $\cat D$, analogous to Grothendieck's local duality  for commutative Gorenstein algebras. The duality on $\cat D$ also involves a (graded-)commutative algebra, namely, $\HH^*(A/k)$, the Hochschild cohomology,  of $A$ over $k$ that acts on $\cat D$ via canonical homomorphisms of $k$-algebras
\begin{flalign*}
&& \HH^*(A/k)\lra \Ext^*_A(X,X)& & \text{for each $X\in \cat D$.}
\end{flalign*} 
In this way $\cat D$ acquires a structure of an $\HH^*(A/k)$-linear category.

In the remainder of the introduction we fix a homogeneous $k$-subalgebra $R$ of  $\HH^*(A/k)$ that is finitely generated as a $k$-algebra.  For simplicity of exposition we assume  that $R^0=k$. This is not a great loss of generality for given any $R$ as above, we can drop down to the subring $k\oplus R^{\ges 1}$ without sacrificing the finite generation. A natural choice for $R$ is $k\oplus \HH^{\geqslant 1}(A/k)$ but, for example, when $A$ is a Hopf algebra, like the group algebra of a finite group, or group scheme, it is more natural to take $R=\Ext^*_A(k,k)$, the cohomology ring of $A$, for this is functorial in the ring argument, whilst the Hochschild cohomology is not.

Fix a homogeneous prime ideal $\fp$ of $R$. Let $\tors_{\fp}(\cat D)$ be the triangulated category obtained from $\cat D$ by localising the graded morphisms at $\fp$ and then taking the full triangulated subcategory of objects such that the graded endomorphisms are $\fp$-torsion. By construction $\tors_{\fp}(\cat D)$ is an $R_\fp$-linear  category, where $R_{\fp}$ denotes the homogenous localisation of $R$ at $\fp$. 
The Nakayama functor induces an equivalence 
\[
\nu_\fp\colon\tors_{\fp}(\cat D)\xra{\sim} \tors_{\fp}(\cat D)\,.
\]

Let $\Proj R$ denote the homogenous prime ideals in $R$ not containing $R^{\geqslant 1}$, the unique maximal homogenous ideal of $R$.  For $\fp\in \Proj R$ the injective hull of the graded $R$-module $R/\fp$ is denoted  $I(\fp)$. Our local Serre duality statement reads:

\begin{theorem}
\label{thm:main1} 
 Let $\fp$ be in $\Proj R$ and  let $d$ be the Krull dimension of $R/\fp$.  Then $ \Sigma^{-d}\circ \nu_\fp$ is a Serre functor for $\tors_{\fp}(\cat D)$, in that,  for all $X,Y$ in $\tors_{\fp}(\cat D)$ there are natural isomorphisms
\[
  \Hom_{R_{\fp}}(\Hom^{*}_{\tors_{\fp}(\cat D)}(X,Y),I(\fp))\cong\Hom_{\tors_{\fp}(\cat D)}(Y,\Sigma^{-d}\nu_\fp (X))\,.
\] 
\end{theorem}
This result is proved towards the end of Section~\ref{sec:Gorenstein} from a more general statement concerning the category $\GProj A$ of (possibly infinite dimensional) Gorenstein projective modules. These are $A$-modules $M$ with the property that $\Ext_A^i(M,P)=0$ for $i\ge 1$ and projective $A$-module $P$. The connection to $\cat D$ is through its subcategory, $\Gproj A$, consisting of finitely dimensional modules. These are  precisely the maximal Cohen-Macaulay $A$-modules, in the terminology of Buchweitz~\cite{Buchweitz:1987}. 

The stable category, $\uGProj A$, of $\GProj A$ is a compactly generated triangulated category, with compact objects equivalent to $\uGproj A$, the stabilisation of $\Gproj A$. Buchweitz~\cite{Buchweitz:1987} proved that there is a equivalence of triangulated categories
\[
\uGproj A\stackrel{\sim}\lra \sing A:= \dbcat{A}/{\cat D}^{\mathrm b}(\proj A)
\]
where $\sing A$ is the \emph{singularity category}, also known as the \emph{stable derived category}, of $A$. There is a natural $R$-action on $\uGProj A$ and the equivalence above is compatible with the induced $R$-actions. For each prime ideal $\fp$ in $\Proj R$ the canonical functors induce equivalences of triangulated categories
\[
\gamma_{\fp}(\uGproj A)\xra{\ \sim\ } \gamma_\fp(\cat{D}) \xra{\ \sim\ } \gamma_{\fp}(\sing A)\
\]
compatible with $R_\fp$-actions, by Lemma~\ref{lem:Tate-local}.  Thus to prove Theorem~\ref{thm:main1} it suffices to prove the corresponding statement for $\uGproj A$; equivalently, for the singularity category of $A$. This also explains the title of this paper.

To that end we consider the subcategory $\gam_{\fp}(\uGProj A)$ of $\uGProj A$ consisting of the $\fp$-local $\fp$-torsion modules. These are the Gorenstein projective $A$-modules $M$ with the property that for each finite dimensional $A$-module $C$, every element of $\uHom^*_{A}(C,M)$, the graded $R$-module of morphisms in $\uGProj A$, is annihilated by some power of $\fp$, and the natural map $\uHom^*_{A}(C,M)\to \uHom^*(C,M)_\fp$ is bijective. Then $\gam_{\fp}(\uGProj A)$  is also a compactly generated triangulated category  and the full subcategory of compact objects is equivalent, up to direct summands, to $\tors_{\fp}(\uGproj A)$. There is an idempotent functor $\gam_\fp\colon \uGProj A\to \uGProj A$ with image the $\fp$-local $\fp$-torsion modules; see Section~\ref{sec:cohomology-and-localisation} for details. The central result of this work is  a local duality theorem  for this category:

\begin{theorem}
\label{thm:main2} 
 Let $\fp$ be in $\Proj R$ and let $d$ be the Krull  dimension of $R/\fp$. Let $X,Y$ be Gorenstein projective $A$-modules and suppose that $X$ is finite dimensional.  Then there is a natural isomorphism
\[ 
  \Hom_R(\Ext_A^*(X,Y),I(\fp))\cong \uHom_A(Y,\Omega^{d}  \gam_{\fp}\GP\nu(X))\,.
\]
\end{theorem}
Here $\GP$ is the Gorenstein projective approximation functor; see Section~\ref{sec:GorAlgebras} for details. The theorem above is contained in Theorem~\ref{thm:gorenstein}.

Theorems~\ref{thm:main1} and \ref{thm:main2} are formulated in terms of an arbitrary (but fixed) subalgebra $R$ of the Hochschild cohomology of $A$. It is thus natural to ask how these results are related as we vary $R$. This point is addressed in Remark~\ref{rem:varyR}. Another issue is what transpires in Theorem~\ref{thm:main1}  if we set $\fp=R^{\geqslant 1}$. When the ring $R$ is such that, in addition to being noetherian, the $R$-module $\Ext^*_A(X,X)$ is finitely generated for each $X\in \cat D$, the subcategory $\gam_\fm(\cat D)$ is  precisely the subcategory of perfect complexes and the analogue of Theorem~\ref{thm:main1} is Happel's duality; see Remark~\ref{rem:fgc}.

 The duality statements above are modeled on, and extensions of, analogous results for representation of modules over finite group schemes established in \cite{Benson/Iyengar/Krause/Pevtsova:2019a}. In that context, the stable category of (finite dimensional) Gorenstein projectives is the stable category of (finite dimensional) representations. We refer to that work for antecedents of these results and for applications, notably,  the existence of AR triangles in $\gamma_{\fp}(\cat {D})$. The proof of the results in \emph{op.~cit.} exploited  the tensor structure on the module categories in question, but in fact the arguments can be readily adapted to deal with the general case, as we do here. In doing so, it became clear that local duality is a feature of Gorenstein algebras in general.

\section{Gorenstein algebras}
\label{sec:GorAlgebras}
Throughout this work $k$ will be a field and $A$ a finite dimensional $k$-algebra that is  Gorenstein  (also known as Iwanaga-Gorenstein): the injective dimension of $A$ as a left $A$-module and as a right $A$-module is finite. In this case, the injective dimensions are the same; this  was proved by Zaks~\cite{Zaks:1969}. Evidently when $A$ is Gorenstein so is $A^\op$, the opposite algebra of $A$. 

In this section we recall basic notions and results from the homological theory of Gorenstein algebras, mainly pertaining to duality. To begin with we note that the Gorenstein condition on $A$ is equivalent to: An $A$-module has finite projective dimension if and only if it has finite injective dimension; see, for example, \cite[Lemma~5.1.1]{Buchweitz:1987}. This will be used often in the sequel.

Let $\Mod A$ be the category all (left) $A$-modules and $\mod A$ its full subcategory consisting of finitely generated $A$-modules. The full subcategory of $\mod A$ consisting of projective $A$-modules is denoted  $\proj A$. The injective analogue is denoted $\inj A$. In what follows for $A$-modules $M,N$  we set
\begin{alignat*}{2}
&\uHom_A(M,N)&&:=\Hom_A(M,N)/\{\phi\mid \phi\text{ factors through a projective}\}\\
&\oHom_A(M,N)&&:=\Hom_A(M,N)/\{\phi\mid \phi\text{ factors through an injective}\}.
 \end{alignat*}
When $M$ and $N$ are finitely generated, it suffices to consider maps $\phi$ that factor through $\proj A$ and $\inj A$, respectively.

In the sequel a \emph{duality} between categories will mean a contravariant equivalence.

\subsection*{Vector space duality} For any $A$-module $M$ we set 
\[
DM:=\Hom_k(M,k)\,,
\] 
viewed as an $A^\op$-module.  The assignment $M\mapsto DM$  induces a duality 
\[
\mod A\xra{\ \sim\ }\mod A^\op\,.
\]
which restricts to  a duality $\proj A \xra{\sim} \inj{A^\op}$. The functor $D$ also extends to a duality between the corresponding bounded derived categories:
\begin{equation}
\label{eq:Dduality-Db}
D\colon {\dbcat A} \xra{\ \sim\ } \dbcat{A^\op}\,.
\end{equation}
Let $\cat{D}^{\mathrm{b}}(\proj A)$ denote the bounded derived category of $\proj A$. We  identify it with the subcategory of $\dbcat A$ consisting of perfect complexes. The duality above restricts to $\cat{D}^{\mathrm{b}}(\proj A)\xra{\sim}\cat{D}^{\mathrm{b}}(\proj A^\op)$, since $A$ is Gorenstein. 

\subsection*{The singularity category}
Buchweitz~\cite{Buchweitz:1987} introduced the \emph{stable derived category} of $A$ as the Verdier quotient
\[
\sing A:=\dbcat{A}/\cat{D}^{\mathrm{b}}(\proj A)\,.
\]
This category was rediscovered by Orlov~\cite{Orlov:2004}, who called it the \emph{singularity category} of $A$, and our notation reflects this terminology. When the global dimension of $A$ is finite one has $\sing A=0$, so the singularity category is one measure of the deviation of $A$ from finite global dimension.
 
The duality \eqref{eq:Dduality-Db} induces a duality 
\begin{equation}
\label{eq:D-duality}
D\colon \sing A \xra{\ \sim\ } \sing{A^{\op}}\,.
\end{equation}

The singularity category can be realized (in more than one way) as a stabilisation of a subcategory of $\mod A$. This is described next.

\subsection*{Gorenstein projective and Gorenstein injective modules} 
An $A$-module $X$ is \emph{Gorenstein projective} if
\begin{flalign*}
&&\Ext_A^i(X,P)=0\quad&\text{for each projective $P$ and each $i\ge 1$.}
\end{flalign*}
We write $\GProj A$ for the category of Gorenstein projective $A$-modules and set
\[
\Gproj A:= \GProj A\cap\mod A\,.
\]
These are the maximal Cohen-Macaulay $A$-modules, in Buchweitz's terminology.

Standard arguments (following, for example,
\cite[Section~9]{Happel:1987}) yield that $\GProj A$ is a Frobenius
exact category, with the projective $A$-modules the projective and
injective objects in this category; this requires the hypothesis that
$A$ is Gorenstein. We write $\uGProj A$ for the corresponding stable
category; it is a triangulated category. Its thick subcategory
consisting of the finitely generated modules is denoted $\uGproj A$,
for it can be realised as a stabilisation of $\Gproj A$.

 On $\Gproj A$ the syzygy functor $\Omega$ has an inverse, denoted $\Omega^{-1}$, that is well-defined up to projective summands.  This is the translation on $\uGProj A$. 

An $A$-module $Y$ is \emph{Gorenstein injective} if
\begin{flalign*}
&&\Ext_A^i(Q,Y)=0\quad&\text{for each injective $Q$ and each $i\ge 1$.}
\end{flalign*}
We write $\GInj A$ for the category of Gorenstein injective modules and set $\Ginj A:=\GInj A\cap \mod A$. The stabilisation of $\GInj A$ is denoted $\oGInj A$, and the subcategory of finitely generated modules is $\oGinj A$. These are triangulated categories, where the translation $\Si Y$ of a $Y\in\GInj A$ is the cokernel of an embedding into an injective $A$-module:
\[
0\lra Y \lra I \lra \Si Y \lra 0 
\]

The duality $D\colon \mod A \xra{\sim} \mod {A^\op}$ restricts to dualities
\begin{equation}
\label{eq:D-Gthings}
\Gproj A \xra{\ \sim\ } \Ginj A^\op\quad\text{and} \quad \proj A \xra{\ \sim\ } \inj A^\op\,.
\end{equation}
and hence induces a duality of triangulated categories
\begin{equation}
\label{eq:D-sGthings}
D\colon \uGproj A \xra{\ \sim\ } \oGinj A^\op\,.
\end{equation}

Here is the result on realising the singularity category as a stabilisation.

\begin{proposition}
\label{pr:buchweitz}
The inclusions $\Gproj A\to \dbcat{A}$ and  $\Ginj A\to \dbcat{A}$ induce triangle equivalences
\[
p\colon \uGproj A\xra{\sim}\sing A\qquad\text{and}\qquad
q\colon \oGinj  A\xra{\sim}\sing  A\,.
  \]
\end{proposition}

\begin{proof} 
The first equivalence is Theorem~4.4.1 from \cite{Buchweitz:1987}, and the  second equivalence follows from that statement applied to $A^\op$ and dualities \eqref{eq:D-sGthings}.
\end{proof}

\subsection*{Approximations}
The following result---see \cite[Lemma~5.1.1]{Buchweitz:1987}---is straightforward to verify.
\begin{lemma}
\label{lem:G-properties}
Let $X$ be a Gorenstein projective module.
\begin{enumerate}[\quad\rm(1)]
\item
$\Ext_A^i(X,F)=0$ for any module $F$ of finite projective dimension and $i\ge 1$.
\item
If an $A$-linear map $X\to N$ factors through a module of finite projective dimension, then it factors through a projective module.
\end{enumerate}
The analogous statements for Gorenstein injective modules also hold.\qed
\end{lemma}

This has the following consequence.

\begin{lemma}
\label{le:uHom=oHom}
If $X$ is Gorenstein projective and $Y$ is Gorenstein injective, then
\[
\uHom_A(X,Y)=\oHom_A(X,Y)\,.
\]
\end{lemma}

\begin{proof}
Given Lemma~\ref{lem:G-properties}, one has to verify that a map $f\colon M\to N$ factors through a module of finite projective dimension if and only if it factors through a module of finite injective dimension. This is a tautology as these categories coincide.
\end{proof}

The following result is due to Auslander and Buchweitz, and is the cornerstone of their theory of maximal Cohen-Macaulay approximations.  There is an analogous statement involving Gorenstein injectives.

\begin{proposition}
\label{pr:approximation}
Every finite dimensional $A$-module $M$ fits into  exact sequences 
\[
0\to F_M\to X_M\to M\to 0\quad \text{and}\quad 0\to M\to F^M\to X^M\to 0
\]
where $X_M,X^M$ are in $\Gproj A$ and $F_M,F^M$ have finite projective dimension. For $X\in\Gproj A$ and $F$ of finite projective dimension, these sequences induce  bijections 
\[
\uHom_A(X,X_M)\xra{\sim}\uHom_A(X,M) \quad \text{and}\quad \uHom_A(F^M,F)\xra{\sim}\uHom_A(M,F)\,.
\]
\end{proposition}

\begin{proof} 
This is part of \cite[Theorems~5.1.2, 5.1.4]{Buchweitz:1987}; see also  \cite[ Theorems~1.8, 2.8]{Auslander/Buchweitz:1989}.
\end{proof}

We write $\GP(M):=X_M$ for the Gorenstein projective  approximation of $M$. It follows from the preceding result that  $X_M$ is well-defined in $\uGproj A$.

\begin{lemma}
\label{le:GP}
The Gorenstein projective approximation $\GP$ induces a triangle equivalence $\oGinj A\xra{\sim}\uGproj A$ satisfying  $q\cong p\circ\GP$, where $p,q$ are the functors in Proposition~\ref{pr:buchweitz}.
\end{lemma}

\begin{proof}
Since $pF=0$ for any $A$-module $F$ of finite projective dimension, Proposition~\ref{pr:approximation} implies $qY\cong p\GP(Y)$ for any Gorenstein injective module $Y$. This yields also that $\GP$ is a triangle equivalence, since $p,q$ are triangle equivalences.
\end{proof}

\begin{remark}
\label{rem:not-self-inj}
When $A$ is self-injective the projective and injective $A$-modules
coincide and hence $\GProj A=\GInj A$. Conversely, when the
(projective) module $A$ is Gorenstein injective, $A$ is
self-injective: Consider an exact sequence
\[
0\lra A\lra I\lra \Si A \lra 0
\]
where $I$ is injective. If $A$ is Gorenstein injective, then so is $\Si A$, and hence the sequence above splits, as the injective dimension of $A$ is finite. Thus $A$ is injective.

Example~\ref{ex:GPI} describes a Gorenstein algebra that is not self-injective, and identifies modules that are Gorenstein projective but not Gorenstein injective.
\end{remark}

\subsection*{The Nakayama functor}
This is the functor $\nu\colon \mod A\to\mod A$ that assigns to each $M$ in $\mod A$ the $A$-module
\[
\nu M:=DA\otimes_A M\cong D\Hom_A(M,A)\,.
\] 
This functor is an equivalence if $A$ is self-injective but not in general. However the Gorenstein property of $A$ implies that the  derived Nakayama functor, for which we use the same notation, is  an equivalence:
\[
\nu\colon\dbcat{A}\stackrel{\sim}\lra \dbcat{A}\quad\text{where}\quad M\mapsto DA \lotimes_A M\,.
\]
From this it is not hard to verify that $\nu\colon \mod A\to \mod A$ restricts to equivalences
\[
\begin{tikzcd}
\proj A \ar[d,hookrightarrow] \ar[r, "\nu","\sim"'] & \inj A \ar[d,hookrightarrow] \\
\Gproj A \ar[r, "\nu","\sim"'] & \Ginj A\,. 
\end{tikzcd}
\]
Therefore one gets an induced equivalence $\nu\colon \uGproj A\xra{\sim} \oGinj A$.

On the other hand, since the complexes of finite projective dimension are the same as those of finite injective dimension (because $A$ is Gorenstein), $DA$ is in the thick subcategory of $\dbcat{A^{\op}}$ generated by $A$, and hence, by duality $A$ is in the thick subcategory of $\dbcat A$ generated by $DA$, that is to say, by $\nu A$. It follows that the derived Nakayama functor satisfies
\[
\nu (\bfD^{\mathrm{b}}(\proj A))=\bfD^{\mathrm{b}}(\proj A)\,.
\]
This is the essence of Happel duality. It induces an equivalence $\bar\nu$ on the singularity category that makes the following square commutative
\begin{equation}
\label{eq:GSnu}
\begin{tikzcd}
    \uGproj A \ar[d,"\sim","p"'] \ar[r,"\nu","\sim"'] & \oGinj A \ar[d,"q","\sim"']  \\
    \sing A \ar[r,"\bar\nu","\sim"'] &\sing A\,.
  \end{tikzcd}
\end{equation}
The vertical equivalences are from Proposition~\ref{pr:buchweitz}.

\subsection*{The Auslander transpose}
Let $X$ be a finite dimensional Gorenstein projective $A$-module. Any projective presentation 
\[
P_1\lra P_0\lra X\lra 0
\]
induces an exact sequence of $A^\op$-modules.
\[
0\lra \Hom_A(X,A)\lra\Hom_A(P_0,A)\lra \Hom_A(P_1,A)\lra \Tr X\lra 0\,.
\]
The $A^\op$-module $\Tr X$ is the \emph{Auslander transpose} of $X$, and it depends only on $X$ up to projective summands. It is  a Gorenstein projective $A^\op$-module since it identifies with $\Hom_A(\Omega^2X,A)$. Applying $D$ yields an exact sequence of $A$-modules
\[
0\lra D\Tr X\lra \nu P_1\lra\nu P_0\lra\nu X\lra 0\,.
\]
Since $\Tr X$ is Gorenstein projective over $A^\op$, the $A$-module $D\Tr X$ is Gorenstein injective and the exact sequence above implies that there is an  isomorphism
\begin{equation}
\label{eq:DTr}
\Si^2(D\Tr X)\cong\nu X.
\end{equation}
in $\oGinj A$. 

\subsection*{Auslander-Reiten Duality}
We are ready to state the main results of this section. For similar statements we refer to \cite{Amiot:2008,Auslander/Reiten:1991}.

\begin{proposition}
\label{pr:serre}
Let $A$ be a Gorenstein algebra, and $X,Y$ Gorenstein projective $A$-modules.
When $X$ is finite dimensional,  there  is a natural isomorphism
\[ 
D\uHom_A(X,Y) \cong\uHom_A(Y,\Omega^{-1}\GP(D\Tr X))\,.
\]
\end{proposition}

\begin{proof}
The exact sequence $0\to \Omega Y\to P\to Y\to 0$ with $P$ projective, induces the first of the following isomorphisms
\begin{align*}
  D\uHom_A(X,Y)&\cong  D\Ext^1_A(X, \Omega Y)\\
               &\cong \oHom_A(\Omega Y,D\Tr X)\\
               &\cong  \uHom_A(\Omega Y,D\Tr X)\\
               &\cong\uHom_A(\Omega Y,\GP(D\Tr X))\\
               &\cong\uHom_A(Y,\Omega^{-1}\GP(D\Tr X))\,.
\end{align*}
The second isomorphism is Auslander-Reiten duality \cite[Proposition I.3.4]{Auslander:1978a}, whilst  the third isomorphism is from Lemma~\ref{le:uHom=oHom}, since $\Omega Y$ is Gorenstein projective and $D\Tr X$ is Gorenstein injective. The fourth  isomorphism follows from Proposition~\ref{pr:approximation}, and the final one is clear since $\Omega$ is an equivalence on $\uGproj A$.
\end{proof}

As in Bondal and Kapranov \cite[\S3]{Bondal/Kapranov:1990a}, a \emph{Serre functor}  on  a $k$-linear, Hom-finite, additive category $\sfC$ is an equivalence $F\colon\sfC\to\sfC$ along with natural isomorphisms
\[
D\Hom_\sfC(X,Y)\cong\Hom_\sfC(Y,FX)
\] 
for all objects $X,Y$ in $\sfC$.

\begin{theorem}
\label{th:serre}
Let $A$ be a Gorenstein algebra. Then
\[
{\Omega\circ\GP}\circ\nu\colon \uGproj A\xra{\sim}\uGproj A \quad\text{and}\quad 
\Si^{-1}\circ \bar\nu\colon \sing A\xra{\sim} \sing A
\] 
are Serre functors.
\end{theorem}
 
\begin{proof}
For any Gorenstein projective $A$-module $X$, one has a sequence of isomorphisms in $\uGproj A$, where the first and the last one are by construction:
\begin{align*}
\Omega^{-1}\GP(D\Tr X) 
	&= \Si \GP(D\Tr X) \\
	& \cong \GP (\Si D\Tr X) \\
	& \cong \GP (\Si^{-1}\nu X) \\
	& \cong \Si^{-1}(\GP \nu X) \\
	&=\Omega(\GP\nu X)\,.
\end{align*}
The second and the fourth isomorphisms hold because $\GP$ is a triangle functor, and the third one is by \eqref{eq:DTr}. Thus  Proposition~\ref{pr:serre} yields that $\Omega\circ\GP\circ\nu$ is a Serre functor on $\uGProj A$, as claimed.  The description of the Serre functor on $\sing A$ is then a consequence of  Lemma~\ref{le:GP} and \eqref{eq:GSnu}.
\end{proof}

\subsection*{Compact generation}
So far the results have mostly dealt with the categories $\uGproj A$ and $\oGinj A$ consisting of finite dimensional modules. To prove the local duality theorem announced in the introduction we need to work in larger categories, $\uGProj A$ and $\oGInj A$. To this end we recall the following result, which is well-known at least for self-injective algebras.

\begin{proposition}
The stable categories $\uGProj A$ and $\oGInj A$ are compactly generated triangulated categories, and the full subcategories of compact objects identify with $\uGproj A$ and $\oGinj A$, respectively.
\end{proposition}

\begin{proof}
The Nakayama functor  induces a triangle equivalence $\uGProj A\xra{\sim}\oGInj A$, identifying $\uGproj A$ with $\oGInj A$; the quasi-inverse is given by $\Hom_A(D(A),-)$. It thus suffices to verify the assertions about $\oGInj A$.

It follows from \cite[\S5]{Iyengar/Krause:2006} that $\oGInj A$ is compactly generated; it remains to identify the compact objects.  If $X$ is a finite dimensional module, then $\oHom_A(X,-)$ preserves direct sums.  Thus every module in $\oGinj A$ is compact in $\oGInj A$. 

On the other hand, for any nonzero module $Y$ in $\oGInj A$ there exists a finite dimensional module $M$ such that $\oHom_A(M,Y)\neq 0$; this is because $Y$ is not injective.  Choose a Gorenstein injective approximation $M\to W$, using the analogue of Proposition~\ref{pr:approximation} for Gorenstein injective modules. Then 
\[
\oHom_A(W,Y)\cong \oHom_A(M,Y)\ne 0\,,
\]
which implies that $\oGinj A$ is all the compact objects of $\oGInj A$.
\end{proof}

\section{Cohomology and localisation}
\label{sec:cohomology-and-localisation} 

In this section we recall basic notions and constructions concerning certain localisation functors on  triangulated categories with ring
actions. The material is needed to state and prove the results in Section~\ref{sec:HH-and-Tate} and \ref{sec:Gorenstein}. The main triangulated category of interest is the stable category of Gorenstein projective modules. Primary references for the material presented here are \cite{Benson/Iyengar/Krause:2008a, Benson/Iyengar/Krause:2011a}.

\subsection*{Triangulated categories with central action} 
Let $\sfT$ be a triangulated category with suspension $\Si$. Given objects $X$ and $Y$ in $\sfT$, consider the graded abelian groups
\[ 
\Hom_\sfT^*(X,Y)=\bigoplus_{i\in\bbZ}\Hom_\sfT(X,\Si^i Y) \quad
\text{and}\quad
\End_{\sfT}^{*}(X)= \Hom_{\sfT}^{*}(X,X)\,.
\] 
Composition makes $\End_{\sfT}^{*}(X)$ a graded ring and $\Hom_{\sfT}^{*}(X,Y)$ a left-$\End_{\sfT}^{*}(Y)$ right-$\End_{\sfT}^{*}(X)$ module.

Let $R$ be a graded-commutative ring. We say the triangulated category $\sfT$ is \emph{$R$-linear} if for each $X$ in $\sfT$ there is a homomorphism of graded rings $\phi_X\colon R\to \End_{\sfT}^{*}(X)$ such that the induced left and right actions of $R$ on $\Hom_{\sfT}^{*}(X,Y)$ are compatible in the following sense: For any $r\in R$ and $\alpha\in\Hom^*_\sfT(X,Y)$, one
has
\[ \phi_Y(r)\alpha=(-1)^{|r||\alpha|}\alpha\phi_X(r)\,.
\]

An exact functor $F\colon \sfT\to\sfU$ between $R$-linear triangulated categories is \emph{$R$-linear} if the induced map
\[ 
\Hom_\sfT^*(X,Y)\lra \Hom_\sfU^*(FX,FY)
\] 
of graded abelian groups is $R$-linear for all objects $X,Y$ in $\sfT$.
 
In what follows, we fix a compactly generated $R$-linear triangulated
category $\sfT$ and write $\sfT^c$ for its full subcategory of compact
objects.

\subsection*{Graded modules} 
In the remainder of this section $R$ will be a graded commutative noetherian ring.  We will only be concerned with homogeneous elements and ideals in $R$. In this spirit, `localisation' will mean homogeneous localisation, and $\Spec R$ will denote the set of homogeneous prime ideals in $R$. 

Given graded $R$-modules $M$ and $N$, we denote by $\Hom_R(M,N)$ the $R$-linear maps $\phi\colon M\to N$ such that
$\phi(M^i)\subseteq N^i$ for all $i\in\bbZ$, and
\[ 
\Hom_R^*(M,N)=\bigoplus_{i\in\bbZ}\Hom_R(M,N[i])
\]
where $N[i]^p=N^{i+p}$ for all $i,p\in\bbZ$.

\subsection*{Localisation} 
Fix an ideal $\fa$ in $R$. An $R$-module $M$ is \emph{$\fa$-torsion} if $M_\fq=0$ for all $\fq$ in $\Spec R$ with $\fa\not\subseteq \fq$. Analogously, an object $X$ in $\sfT$ is \emph{$\fa$-torsion} if the $R$-module $\Hom_\sfT^*(C,X)$ is $\fa$-torsion for all $C\in\sfT^c$.  The full subcategory of $\fa$-torsion objects
\[ 
\gam_{\mcV(\fa)}\sfT:=\{X\in\sfT\mid X \text{ is $\fa$-torsion} \}
\] 
is localising and the inclusion $\gam_{\mcV(\fa)}\sfT\subseteq \sfT$ admits a right adjoint, denoted $\gam_{\mcV(\fa)}$.

Fix a $\fp$ in $\Spec R$.  An $R$-module $M$ is \emph{$\fp$-local} if the localisation map $M\to M_\fp$ is invertible, and an object $X$ in $\sfT$ is \emph{$\fp$-local} if the $R$-module $\Hom_\sfT^*(C,X)$ is $\fp$-local for all $C\in\sfT^c$.  Consider the full subcategory of $\sfT$ of $\fp$-local objects
\[ 
\sfT_\fp:=\{X\in\sfT\mid X \text{ is $\fp$-local}\}
\] 
and the full subcategory of $\fp$-local and $\fp$-torsion objects
\[ 
\gam_\fp\sfT:=\{X\in\sfT\mid X \text{ is $\fp$-local and
$\fp$-torsion} \}.
\] 
Note that $\gam_\fp\sfT\subseteq\sfT_\fp\subseteq\sfT$ are localising subcategories.  The inclusion $\sfT_\fp\to\sfT$ admits a left adjoint $X\mapsto X_\fp$ while the inclusion $\gam_\fp\sfT\to\sfT_\fp$ admits a right adjoint. We denote by $\gam_\fp\colon\sfT\to\gam_\fp\sfT$ the composition of those adjoints;
it is the \emph{local cohomology functor} with respect to $\fp$; see \cite{Benson/Iyengar/Krause:2008a, Benson/Iyengar/Krause:2011a} for explained notions and details.

The functor $\gam_{\mcV(\fa)}$ commutes with exact functors preserving coproducts.

\begin{lemma}
\label{le:gamma-commute} 
Let $F\colon\sfT\to \sfU$ be an exact functor between $R$-linear compactly generated triangulated categories such that $F$ is $R$-linear and preserves coproducts. Suppose that the action of $R$ on $\sfU$ factors through a homomorphism $f\colon R\to S$ of graded commutative rings. For any ideal $\fa$ of $R$ there is a natural isomorphism
\[ 
F\comp\gam_{\mcV(\fa)}\cong \gam_{\mcV(\fa S)} \comp F
\] 
of functors $\sfT\to\sfU$, where $\fa S$ denotes the ideal of $S$ that is generated by $f(\fa)$.
\end{lemma}

\begin{proof} 
The statement follows from an explicit description of  $\gam_{\mcV(\fa)}$ in terms of homotopy colimits; see \cite[Proposition~2.9]{Benson/Iyengar/Krause:2011a}.
\end{proof}

The following observation is clear.

\begin{lemma}
\label{le:periodicity}
For any element $r$ in $R\setminus \fp$, say of degree $n$, and $\fp$-local object $X$, the natural map $X\xra{r}\Si^n X$ is an isomorphism. \qed
\end{lemma}

\subsection*{Koszul objects} 

Fix objects $X,Y$ in $\sfT$. Any element $b$ in $R^d$  induces a morphism $X\to \Sigma^dX$ and let $\kos X {b}$ denote its mapping cone. This gives a morphism $X\to\Sigma^{-d}(\kos X {b})$. For a sequence of elements $\bsb:=b_{1},\dots,b_{n}$ in $R$  set $\kos X{\bsb}:= X_n$ where
\[
X_0:=X\quad\text{and} \quad X_i:=\kos{X_{i-1}}{b_i} \quad\text{for}  \quad 1\le i\le n\,.
\]
It is easy to check that for  $s=\sum_{i}|b_{i}|$ there is an isomorphism
\begin{equation}
\label{eq:kos-swap} 
\Hom_\sfT(X,\kos Y{\bsb}) \cong \Hom_\sfT(\kos X{\bsb},\Sigma^{s+n}Y).
\end{equation}

\subsection*{Injective cohomology objects}

Given an object $C$ in $\sfT^{\sfc}$ and an injective $R$-module $I$,
Brown representability yields an object $T(C,I)$ in $\sfT$ such that
\begin{equation}\label{eq:inj-coh} 
\Hom^*_R(\Hom^*_\sfT(C,-),I)\cong\Hom^*_\sfT(-,T(C,I))\,.
\end{equation} 
This yields a functor
\[ 
T\colon \sfT^{\sfc}\times\Inj R\lra\sfT.
\] 
For each $\fp$ in $\Spec R$, we write $I(\fp)$ for the injective hull of $R/\fp$ and set
\[ 
T_\fp:=T(-,I(\fp))\,,
\] 
viewed as a functor $\sfT^{\sfc}\to \sfT$. For objects $C$ and $D$ in $\sfT^{\sfc}$, applying \eqref{eq:inj-coh} twice one gets a natural $R$-linear isomorphism
\begin{equation}\label{eq:inj-coh-hom}
\Hom_\sfT^*(T(C,I),T(D,I))\cong\Hom^*_R(\Hom^*_R(\Hom^*_\sfT(C,D),I),I)\,.
\end{equation}

\section{Hochschild and Tate cohomology}
\label{sec:HH-and-Tate}
Let $k$ be a field and $A$ a finite dimensional Gorenstein $k$-algebra. The \emph{enveloping algebra} of $A$ is the $k$-algebra $A^e:=A\otimes_k A^\op$; it is also Gorenstein, by Proposition~\ref{pr:gor-tensor}, but this observation does not play a role in the sequel. The \emph{Hochschild cohomology} of the $k$-algebra $A$ is 
\[
\HH^*(A/k):=\Ext_{A^e}^*(A,A)\,.
\]
This is a graded-commutative $k$-algebra.  

When $X$ and $Y$ are Gorenstein projective $A$-modules,  we set 
\[
\uHom_A^*(X,Y)=\bigoplus_{i\in\bbZ}\uHom_A^*(X,\Omega^{-i}Y)\,.
\]
This is the \emph{Tate cohomology} of $X,Y$. There is a canonical homomorphism
\begin{equation}
\label{eq:ext-tate}
\Ext^*_A(X,Y)\lra \uHom_A^*(X,Y),
\end{equation}
of graded abelian groups, induced from the canonical morphism $\bst X\to\bsp X$ from a complete projective resolution to a projective resolution of $X$; see \cite[6.2]{Buchweitz:1987}. In particular, this map is surjective in degree $0$ and bijective in positive degrees.

\subsection*{Action on $\GProj A$}
For any $A$-module $M$ there is a canonical map
\[
 \HH^*(A/k)\xra{-\otimes_AM} \Ext^*_A(M,M)
\]
that is a morphism of graded $k$-algebras. When $X$ is Gorenstein projective, composing the map above with the one in \eqref{eq:ext-tate} one gets a homomorphism of $k$-algebras
\[
\phi_X\colon\HH^*(A/k)\lra \uHom_A^*(X,X)
\]
and this induces a linear action of $\HH^*(A/k) $ on $\uGProj A$, in the sense of Section~\ref{sec:cohomology-and-localisation}.

\begin{assumption}
\label{ass:R}
We fix a homogenous $k$-subalgebra $R$ of $\HH^*(A/k)$ such that
\begin{enumerate}[\quad\rm(1)]
\item
$R$ is \emph{connected}, that is to say, $R^0=k$;
\item
$R$ is finitely generated as a $k$-algebra.
\end{enumerate}
Connectedness implies that $R^{\ge 1}$ is the unique maximal ideal of $R$, which allows us to import the results from \cite{Benson/Iyengar/Krause/Pevtsova:2019a}. This is not a serious restriction. Indeed $A$ decomposes as a direct product of connected algebras, and for a connected algebra $A$ the ring $\HH^0(A/k)$, being the center of $A$, is a finite dimensional local ring, so the inclusion $R^0\subseteq \HH^0(A/k)$ induces a bijection on the spectra. 

As usual  $\Proj R$ denotes the set of prime ideals that do not contain $R^{\geqslant 1}$.
 
Condition (2) is equivalent to the condition that the ring $R$ is noetherian; see \cite[Proposition~1.5.4]{Bruns/Herzog:1998a}. Since the $\HH^{*}(A/K)$-action on  $\uGProj A$ restricts to  an $R$-action,  the noetherian property of $R$ allows one to invoke the constructions and results presented in Section~\ref{sec:cohomology-and-localisation}.
\end{assumption}

\subsection*{Base change}
Let $K/k$ be an extension of fields and set 
\[
A_K:=K\otimes_kA\,.
\]
This is a finite dimensional Gorenstein $K$-algebra, and extension of scalars 
\[
\Mod A\lra \Mod A_K,\quad X\mapsto X_K= K\otimes_{k}X
\]
and restriction 
\[
\Mod A_K\lra \Mod A, \quad X\mapsto  X\da_{A}=\Hom_K(K,X) 
\] 
form an adjoint pair of exact functors.

\begin{lemma}
For $A$-modules $X$ and $Y$, the canonical $K$-linear map 
\[
K\otimes_k\Hom_A(X,Y)\lra\Hom_{A_K}(X_K,Y_K)
\] 
is a isomorphism when $X$ is finite dimensional over $k$.\qed
\end{lemma}

\begin{lemma}
\label{le:base-change}
Extension and restriction preserve projectivity and injectivity. Thus $A$ is Gorenstein if and only if $A_K$ is Gorenstein. In that case extension and restriction preserve Gorenstein projectivity.\qed
\end{lemma}

There are isomorphisms of $K$-algebras
\[
(A_K)^e\cong(A^e)_K \quad\text{and}\quad  \HH^{*}(A_{K}/K)\cong K\otimes_{k} \HH^{*}(A/k)\,.
\]
Thus extension of scalars yields a ring homomorphism $\HH^{*}(A/k)\to \HH^{*}(A_{K}/K)$, and setting $R_K:=K\otimes_kR$ one gets a ring homomorphism
\[
R\lra R_K\,.
\]
Observe that $R_K$ is connected and a noetherian $K$-subalgebra of $\HH^*(A_{K}/K)$.

Next we recall a construction from \cite{Benson/Iyengar/Krause/Pevtsova:2018} of a field extension $K/k$ and a closed point in $\Proj R_K$ lying over a given a point in $\Proj R$.

\begin{construction}
\label{con:generic} 
Let $R$ be a finitely generated, graded, connected, $k$-algebra. Fix a point $\fp$ in $\Proj R$, and let $d$ be the Krull dimension of $R/\fp$.

Choose elements $\bsa:=a_{0},\dots,a_{d-1}$ in $R$ of the same degree such that their image in $R/\fp$ is algebraically independent over $k$ and $R/\fp$ is finitely generated as a module over the subalgebra $k[\bsa]$. Set $K:=k(t_{1},\dots,t_{d-1})$, the field of rational functions in indeterminates $t_{1},\dots,t_{d-1}$ and
\[ 
b_{i}:= a_{i} - a_{0}t_{i}\quad\text{for $i=1,\dots,d-1$}
\] 
viewed as elements in $R_K$. Let $\fp'$ denote the extension of $\fp$ to $R_K$, and set
\begin{equation*}
\label{eq:generic-point} 
\fq:= \fp' + (\bsb)\qquad\text{and}\qquad \fm:=\sqrt \fq\,.
\end{equation*} 
The following statements hold:
\begin{enumerate}[\quad\rm(1)]
\item
$\fm$ is a closed point in $\Proj R_K$ with the property that $\fm\cap R=\fp$;
\item
the  induced extension of fields $k(\fp)\xra{\ \cong\ } k(\fm)$  is an isomorphism.
\end{enumerate}
The first part is contained in \cite[Theorem~7.7]{Benson/Iyengar/Krause/Pevtsova:2018}. The second one holds by construction; see \cite[Lemma~7.6, and (7.2)]{Benson/Iyengar/Krause/Pevtsova:2018}.

 Fix an object $X$ in $\uGProj A$. The sequence of elements  $\bsb$ in $R_K$ yields a morphism  $X_K\to \Omega^{s} (\kos {X_K}{\bsb})$, where $s=\sum_{i}|b_{i}|$. Composing its restriction to $A$ with the canonical morphism
 $X\to (X_K)\da_{A}$ gives in $\uGProj A$ a morphism
\[ 
f_X\colon X \lra \Omega^{s} (\kos {X_K}{\bsb})\da_A\,.
\] 
Since the $a_{i}$ are not in $\fp$, when $X$ is $\fp$-local Lemma~\ref{le:periodicity} yields a natural  isomorphism
\begin{equation}
\label{eq:periodicity} \Omega^{s}X \cong X
\end{equation} 
in $\uGProj A$.
\end{construction}

The result below extends \cite[Theorem~3.4]{Benson/Iyengar/Krause/Pevtsova:2019a} that concerns  modules over finite group schemes, but the argument is essentially the same.

\begin{theorem}
\label{thm:realisability}
For any Gorenstein projective $A$-module $X$, the morphism $f_X$ induces a natural  isomorphism
\[ 
\gam_{\fp}X\cong \gam_\fm(\kos{X_K}{\bsb})\da_A\,.
\] 
 in $\uGProj A$. When $X$ is $\fp$-torsion, this induces a natural isomorphism
\[ 
\gam_{\fp}X \cong (\kos {X_K}{\bsb})\da_A \,.
\]
\end{theorem}

\begin{proof} 
Given the second isomorphism, the first one can be checked as follows:
Let $X$ be a Gorenstein projective $A$-module and $\fp'$ the ideal in
Construction~\ref{con:generic}. Since $\gam_{\mcV(\fp)}X$ is
$\fp$-torsion, one gets the second isomorphism below:
\begin{align*} 
\gam_{\fp}X&\cong (\gam_{\mcV(\fp)} X)_\fp\\ & \cong
(\kos{(\gam_{\mcV(\fp)} X)_K}{\bsb})\da_A\\ & \cong  
(\kos{\gam_{\mcV(\fp')} (X_K)}{\bsb})\da_A\\ & \cong  
(\gam_{\mcV(\fp')}(\kos{X_K}{\bsb}))\da_A\\ & \cong  
(\gam_{\mcV(\fp'+(\bsb))}(\kos{X_K}{\bsb}))\da_A\\ & \cong  
(\gam_\fm(\kos{X_K}{\bsb}))\da_A\,.
\end{align*} 
The third one is by Lemma~\ref{le:gamma-commute}, applied to the functor $K\otimes_{k}(-)$ from $\uGProj A$ to $\uGProj A_{K}$. The next
one is standard while the penultimate one holds because $\kos{X_K}{\bsb}$
is $(\bsb)$-torsion. 

It remains to verify the second isomorphism in the statement. The modules satisfying this isomorphism form a localising subcategory of $\uGProj A$. Moreover,   by \cite[Proposition~2.7]{Benson/Iyengar/Krause:2011a},  the $\fp$-torsion modules form a localising subcategory of $\uGProj A$ generated by the modules $\kos X{\fp}$, for $X\in\Gproj A$. It thus suffices to verify the desired isomorphism for such modules. Since $\kos X{\fp}$ is $\fp$-torsion, the natural map 
\[
\gam_{\fp}(\kos X{\fp})=\gam_{\mcV(\fp)}(\kos X{\fp})_\fp\to (\kos X{\fp})_{\fp}
\]
is an isomorphism. Thus the task reduces to verifying that $f_{(\kos X{\fp})_\fp}$ is an isomorphism for each $X\in\Gproj A$. This is proved in \cite[Theorem~8.8]{Benson/Iyengar/Krause/Pevtsova:2018} for the case of finite group schemes and  $X=k$, the trivial representation. However the argument only uses \cite[Proposition~6.2(2)]{Benson/Iyengar/Krause/Pevtsova:2018} which in turn is a formal consequence of Lemma~\ref{le:gamma-commute}, and thus carries over to the present context.
\end{proof}

\section{The Gorenstein property}
\label{sec:Gorenstein}
Let $F\colon\uGproj A\to \uGproj A$ be the Serre functor from Theorem~\ref{th:serre}, given by
\[
F(X)=(\Omega \circ\GP\circ \,\nu) (X).
\] 
Given the description of $F$, the result below contains Theorem~\ref{thm:main2}.  It extends \cite[Theorem~5.1]{Benson/Iyengar/Krause/Pevtsova:2019a}, which deals with the case $A$ is the group algebra of a finite group scheme and $R=\Ext^*_A(k,k)$, its cohomology ring.  

\begin{theorem}
  \label{thm:gorenstein}
Let $k$ be a field, $A$ a finite dimensional, Gorenstein, $k$-algebra and $R\subseteq \HH^*(R/k)$ as in \ref{ass:R}.  Fix $\fp\in\Proj R$, and let $d$ be the Krull dimension of $R/\fp$. On $\uGproj A$ there is  a natural isomorphism of functors
\[ 
\gam_{\fp}\comp F\cong \Omega^{-d+1}\comp T_{\fp}\,.
\] 
Thus for any object $X$ in $\Gproj A$ there is a natural isomorphism
\[ 
\uHom_A(-,\Omega^{d-1} \gam_{\fp}F(X)) \cong\Hom_{R}(\Ext_A^*(X,-),I(\fp))\,.
\]
\end{theorem}
This result is proved further below,  following some preparatory remarks.

\begin{remark}
  \label{re:Tate-coh} 
Let $X,Y$ be Gorenstein projective $A$-modules. The natural map 
\[
\Ext^{*}_A(X,Y)\to\uHom^{*}_A(X,Y)
\]
is compatible with action of $\HH^*(A/k)$, and hence of $R$. The map is surjective in degree zero, with kernel $\PHom_A(X,Y)$, the maps from $X$ to $Y$ that factor through a projective $A$-module. Since it is bijective in positive degrees one gets an exact sequence of graded $R$-modules
\begin{equation}
\label{eq:Tate-coh} 
0\lra \PHom_A(X,Y) \lra \Ext^{*}_A(X,Y)\lra \uHom^{*}_A(X,Y) \lra C\lra 0
\end{equation} 
with $C^{i}=0$ for $i\ge 0$. For degree reasons, the $R$-modules $\PHom_A(X,Y)$ and $C$ are $R^{\geqslant 1}$-torsion so for $\fp$ in $\Proj R$ the induced localised map is an isomorphism:
\begin{equation}
\label{eq:Tate-local}
\Ext^{*}_A(X,Y)_{\fp}\xra{\sim} \uHom^{*}_A(X,Y)_{\fp}\,.
\end{equation} 
\end{remark}

It follows from Lemma~\ref{le:base-change} that the  functor $F$ is compatible with base change.

\begin{lemma}
\label{le:serre-extension}
Let $K/k$ be a field extension and  $F_K\colon\uGproj A_K\to \uGproj A_K$ the corresponding Serre functor. For  $X$ in $\Gproj A$
there is a natural isomorphism $F_K(X_K)\cong F(X)_K$. \qed
\end{lemma}

The argument  below is direct adaptation of the one for \cite[Theorem~5.1]{Benson/Iyengar/Krause/Pevtsova:2019a}.

\begin{proof}[Proof of Theorem~\ref{thm:gorenstein}]
The proof uses the following  observation: For any $A$-modules $X,Y$ that are $\fp$-local and $\fp$-torsion, there is an isomorphism $X\cong Y$ in $\uGProj A$ if and only if there is a natural isomorphism
\[ 
\uHom_A(M,X)\cong \uHom_A(M,Y)
\] 
for $\fp$-local and $\fp$-torsion $A$-modules $M$. This follows from Yoneda's lemma. 

In anticipation of using the preceding remark, we note that $\gam_\fp(X)$ and $T_\fp(X)$ are $\fp$-local and $\fp$-torsion. This is clear for $\gam_\fp(X)$ and follows for $T_\fp(X)$ from the fact that $I(\fp)$ is a $\fp$-local and $\fp$-torsion $R$-module. Another observation is that, by \eqref{eq:Tate-local}, for any $\fp$-local $R$-module $I$, there is an isomorphism
\[ 
\Hom_{R}(\Ext^*_A(X,-),I)\cong \Hom_{R}(\uHom_A^*(X,-),I).
\] 
Consequently, one can rephrase the defining isomorphism \eqref{eq:inj-coh} for the object $T_\fp(X)$ as a natural isomorphism
\[ 
\uHom_A(-,T_\fp(X) )\cong\Hom_{R}(\Ext_A^*(X,-),I(\fp))\,.
\]
Our task is to verify that, on $\uGproj A$, there is an isomorphism of functors 
\[
\gam_{\fp}F \cong \Omega^{-d+1}T_\fp\,.
\]
 We verify this when $\fp$ is  closed and then use a reduction to closed points. 

\begin{claim} 
  The desired isomorphism holds when $\fm$ is a closed point in $\Proj R$.
\end{claim}

The injective hull, $I(\fm)$, of the $R$-module $R/\fm$ is the same as that of the $R_\fm$-module $k(\fm)$, viewed as an $R$-module via restriction of scalars along the localisation map $R\to R_\fm$. Let $Y$ be a Gorenstein projective $A$-module that is $\fm$-local and $\fm$-torsion. The claim is a consequence of the following
computation:
\begin{align*} 
\uHom_{A}(Y, \gam_{\fm}(FX)) & \cong \uHom_{A}(Y, FX) \\ & \cong \uHom_{A}(X,Y)^{\vee} \\ & \cong
\Hom_{R_\fm}(\uHom_{A}^{*}(X,Y),I(\fm)) \\ & \cong
\Hom_{R}(\uHom_{A}^{*}(X,Y),I(\fm)) \\ & \cong \uHom_A(Y,T_\fm(X))\,.
\end{align*} 
The first isomorphism holds because $Y$ is $\fm$-torsion; the second is Serre duality, Proposition~\ref{pr:serre}, and the next one is by
\cite[Lemma~A.2]{Benson/Iyengar/Krause/Pevtsova:2019a}, which applies because $\uHom_{A}^{*}(X,Y)$ is $\fm$-local and $\fm$-torsion as an $A$-module.

\medskip

Let $\fp$ be a point in $\Proj R$ that is not closed, and let $K$, $\bsb$, and $\fm$ be as in Construction~\ref{con:generic}. Recall that $\fm$ is a closed point in $R_K$ lying over $\fp$.

\begin{claim} 
In $\uGProj A$ there is an  isomorphism of $A$-modules
\begin{equation}
\label{eq:claim-T-module} 
(\kos {T_{\fm}(X_K)}{\bsb})\da_A \cong \Omega^{-d+1} T_{\fp}(X)
\end{equation} 
where $d$ is the Krull dimension of $R/\fp$.
\end{claim}

Let $Y$ be an $A$-module that is $\fp$-local and $\fp$-torsion. Then we
have the following:
\begin{align*} 
  \uHom_{A}(Y, \Omega^{d-1}(\kos {T_{\fm}(X_K)}{\bsb})\da_A) 
  &\cong\uHom_{A_{K}}(Y_{K},\Omega^{d-1}(\kos {T_{\fm}(X_K)}{\bsb})) \\ 
  &\cong\uHom_{A_{K}}(\kos{Y_{K}}{\bsb},T_{\fm}(X_K)) \\ 
  &\cong\Hom_{R_K}(\uHom_{A_{K}}^{*}(X_K,\kos{Y_{K}}{\bsb}),I(\fm))\\ 
  &\cong \Hom_{R}(\uHom_{A_{K}}^{*}(X_K,\kos{Y_{K}}{\bsb}),I(\fp))\\ 
  &\cong\Hom_{R}(\uHom_{A}^{*}(X,(\kos{Y_{K}}{\bsb})\da_A),I(\fp))\\ 
  &\cong \Hom_{R}(\uHom_{A}^{*}(X,Y),I(\fp)) \\ 
  &\cong\uHom_A(Y, T_\fp(X))\,.
\end{align*} 
The first and  fifth isomorphisms are by adjunction. The second one is a direct computation using \eqref{eq:kos-swap} and \eqref{eq:periodicity}. The next one is by definition and the fourth isomorphism is by \cite[Lemma~A.3]{Benson/Iyengar/Krause/Pevtsova:2019a}, applied to the  homomorphism $R\to R_K$; it applies as the $R_K$-module $\uHom_{A_{K}}^{*}(X_K, \kos{Y_{K}}{\bsb})$ is $\fm$-torsion. The sixth isomorphism is by Theorem~\ref{thm:realisability}, and the last one  by definition.  This justifies the claim. 

\medskip

Consider now the chain of isomorphisms:
\begin{align*}
\gam_\fp F(X)&\cong\gam_\fm (\kos{F(X)_K}{\bsb})\da_A\\
&\cong\gam_\fm (\kos{F_K(X_K)}{\bsb})\da_A\\
&\cong(\kos{\gam_\fm (F_K(X_K))}{\bsb})\da_A\\
&\cong (\kos{T_\fm (X_K)}{\bsb})\da_A\\
&\cong \Omega^{-d+1}T_\fp(X).
\end{align*}
The first isomorphism is by Theorem~\ref{thm:realisability}, the second by Lemma~\ref{le:serre-extension}, the third is clear, the
fourth follows from the first claim, since $\fm$ is a closed point for $A_K$, and the last one follows from the second claim.

This completes the proof that the functors $\gam_{\fp}\circ F$ and $\Omega^{-d+1}\circ T_\fp$ are isomorphic. Given this and the alternative description of $T_\fp$ above, the last isomorphism in the statement follows.
\end{proof}

Next we record a corollary of Theorem~\ref{thm:gorenstein} concerning $\tors_{\fp}(\uGproj A)$, the  $\fp$-torsion objects in the $\fp$-localisation of $\uGproj A$; see \cite[\S7]{Benson/Iyengar/Krause/Pevtsova:2019a} and the references therein for details of this construction. The $R$-linear triangle equivalences
\[
\nu\colon \uGproj A\to \oGinj A\quad\text{and}\quad \GP\colon \oGinj A\to \uGproj A
\]
induce $R_\fp$-linear triangle equivalences 
\begin{equation}
\label{eq:tors-diagram}
\begin{gathered}
\begin{tikzcd}
\tors_{\fp}(\uGproj A) \ar[d,hookrightarrow] \ar[r,"\nu_\fp", "\sim" swap] & \tors_{\fp}(\oGinj A)\ar[d,hookrightarrow] \\
(\uGproj A)_\fp \ar[r,"\nu_\fp", "\sim" swap] & (\oGinj A)_\fp 
\end{tikzcd}
\end{gathered}
\quad\text{and}\quad
\begin{gathered}
\begin{tikzcd}
\tors_{\fp}(\oGinj A) \ar[d,hookrightarrow] \ar[r,"\GP_{\fp}", "\sim" swap] & \tors_{\fp}(\uGproj A)\ar[d,hookrightarrow] \\
(\oGinj A)_\fp \ar[r,"\GP_{\fp}", "\sim" swap] & (\uGproj A)_\fp 
\end{tikzcd}
\end{gathered}
\end{equation}
compatible with the localisation functor; see \cite[Remark~7.1]{Benson/Iyengar/Krause/Pevtsova:2019a}. The  result below can be interpreted as the statement that  the category $\tors_{\fp}(\uGproj A)$, and hence also $\tors_{\fp}(\oGinj A)$, has a Serre functor.

\begin{corollary}
\label{cor:Serre-duality-GP}
For $X,Y$ in $\cat{C}:=\tors_{\fp}(\uGproj A)$, there is a natural isomorphism
\[
\Hom_{R_\fp}(\uHom^*_{\cat{C}}(X,Y),I(\fp)) \cong \uHom_{\cat{C}}(Y,\Omega^{d}\GP_{\fp}\nu_{\fp}(X))\,.
\]
\end{corollary}

\begin{proof}
Up to direct summands, the categories $\cat{C}$ and $\gam_{\fp}(\uGProj A)^{\sfc}$ are equivalent, and the compact objects in $\gam_{\fp}(\uGProj A)$ are of the form $M_\fp$, for some $M\in \uGproj A$ that is $\fp$-torsion; see \cite[Remark~7.2]{Benson/Iyengar/Krause/Pevtsova:2019a}. We obtain the desired isomorphism by reinterpreting the isomorphisms in Theorem~\ref{thm:gorenstein}, as follows.

Fix $\fp$-torsion objects $M,N$ in $\uGproj A$. One has the first isomorphism below because $M$ is finite dimensional:
\begin{align*}
\Hom_{R_\fp}(\uHom^*_{A}(M_\fp,N_\fp),I(\fp)) 
	&\cong \Hom_{R_\fp}(\uHom^*_{A}(M,N)_{\fp},I(\fp)) \\
	&\cong \Hom_{R_\fp}(\Ext^*_{A}(M,N)_{\fp},I(\fp))\\
	&\cong \Hom_{R}(\Ext^*_{A}(M,N),I(\fp))\,.
\end{align*}
The second one is by \eqref{eq:Tate-local}, and the last one is by adjunction. 

On the other hand, since $M$ is $\fp$-torsion, so is $F(M)$ and hence one has the first isomorphism below:
\[
\Omega^{d-1} \gam_{\fp}F(M)\cong \Omega^{d-1} F(M)_\fp \cong \Omega^{d-1}(\Omega \GP \nu(M))_\fp 
	\cong \Omega^{d}\GP_\fp\nu_\fp(M_\fp)
\]
The second one is by the definition of $F$ and the third one follows by the discussion around \eqref{eq:tors-diagram}.  Applying $\uHom_{A}(N,-)$ to the composition yields the first isomorphism below, whilst the second one holds as the covariant argument is $\fp$-local:
\begin{align*}
  \uHom_{A}(N_\fp,\Omega^{d}\GP_{\fp}\nu_{\fp}(M_\fp))
  &\cong \uHom_{A}(N_\fp,\Omega^{d-1}\gam_{\fp}F(M)) \\
  &\cong \uHom_{A}(N,\Omega^{d-1}\gam_{\fp}F(M))\,.
\end{align*}
The isomorphisms above and Theorem~\ref{thm:gorenstein}, applied with $X=M$ and $Y=N$, yield the desired result.
\end{proof}

\begin{lemma}
\label{lem:Tate-local}
For any $\fp\in \Proj R$ the quotient functor $\dbcat A\to \sing A$ induces equivalences
\[
{\dbcat A}_\fp \xra{\ \sim\ } {\sing A}_\fp \quad\text{and}\quad \tors_{\fp}(\dbcat A) \xra{\ \sim\ } \tors_{\fp}(\sing A)
\]
compatible with the $R_\fp$ actions.
\end{lemma}

\begin{proof}
The quotient functor is essentially surjective and hence so is its $\fp$-localisation ${\dbcat A}_\fp \to {\sing A}_\fp$. The latter is also fully faithful, by \eqref{eq:Tate-local}, and so is an equivalence of categories. It is also $R_\fp$-linear, by construction. Given this, it is immediate from the definition that the $\fp$-torsion subcategories are equivalent. 
\end{proof}

\begin{proof}[Proof of Theorem~\ref{thm:main1}]
The hypotheses is that $A$ is a finite dimensional Gorenstein algebra and $R$ is a finitely generated, homogenous, $k$-subalgebra of $\HH^*(A/k)$ with $R^0=k$.  Also, $\cat{D}:=\dbcat A$. We fix a homogeneous prime ideal $\fp$ in $\Proj R$. We want to verify that for all $X,Y$ in $\tors_{\fp}(\cat D)$ there are natural isomorphisms
\[
  \Hom_{R_{\fp}}(\Hom^{*}_{\tors_{\fp}(\cat D)}(X,Y),I(\fp))\cong\Hom_{\tors_{\fp}(\cat D)}(Y,\Sigma^{-d}\nu_\fp (X))
\] 
with $d$ the Krull dimension of $R/\fp$. Lemma~\ref{lem:Tate-local} gives the second equivalence below:
\[
\gamma_{\fp}(\uGproj A)\xra{\ \sim\ } \gamma_\fp(\cat{D}) \xra{\ \sim\ } \gamma_{\fp}(\sing A)\,,
\]
whereas the first one is induced by Proposition~\ref{pr:buchweitz},
and both these are compatible with the $R_\fp$ actions.  It remains to recall Corollary~\ref{cor:Serre-duality-GP}.
\end{proof}

\subsection*{The \fgc condition}
As in \eqref{ass:R}, let $R\subseteq \HH^*(A/k)$ be a connected, noetherian $k$-subalgebra. Assume in addition that for each $M\in \mod A$ the $R$-module  $\Ext^*_A(M,M)$ is finitely generated. Said otherwise, the algebra $A$ satisfies the \fgc condition with respect to $R$, introduced in \cite{Erdmann/Holloway/Taillefer/Snashall/Solberg:2004}. As noted in \cite[Proposition~5.7]{Solberg:2006}, this condition implies that the Hochschild cohomology algebra $\HH^*(A/k)$ is finitely generated. The \fgc condition holds for several interesting classes of finite dimensional Hopf algebras, including group algebras of finite groups (or group schemes),  small quantum groups,  and also for finite dimensional commutative complete intersection rings; see \cite{Erdmann/Holloway/Taillefer/Snashall/Solberg:2004, Ginzburg/Kumar:1993, Mastnak/Pevtsova/Schaunberg/Witherspoon:2010, Solberg:2006}.

When $A$ satisfies the \fgc condition, it follows from Corollary~\ref{cor:Serre-duality-GP} that  the triangulated category $\tors_{\fp}(\uGproj A)$, and hence also anything equivalent to it, has AR-triangles. This is explained in \cite[Section~7]{Benson/Iyengar/Krause/Pevtsova:2019a}, for which we refer the reader also for other consequences of Serre duality.

\begin{remark}
\label{rem:fgc}
 Assume that the algebra $A$ satisfies the \fgc condition with respect to $R$ and set $\fr:=R^{\geqslant 1}$, the homogenous maximal ideal of $R$. Set $\cat D:=\dbcat A$.

\begin{claim}
The triangulated category $\tors_{\fr}(\cat D)$ is equivalent to  $\cat{D}^{\sfb}(\proj A)$.
\end{claim}

Indeed, by definition, $\tors_\fr(\cat{D})$  consists of complexes $M\in\cat{D}$ for which $\Ext_R^*(M,M)$ is $\fr$-torsion as an $R$-module; equivalently,  $\Ext_R^*(M,N)$ is $\fr$-torsion for each $N\in\cat{D}$. In particular, $\Ext_R^*(M,A/J)$ is $\fr$-torsion, where $J$ is the Jacobson radical of $A$. Since the $R$-module $\Ext_R^*(M,A/J)$ is finitely generated, by the \fgc condition, this last property is equivalent to  $\Ext_R^*(M,A/J)=0$ for $i\gg 0$, that is to say,  $M$ is perfect.

Given this claim, one can extend Theorem~\ref{thm:main1} even to $\fr$: For perfect complexes $X,Y$ from~\cite{Happel:1987} we get the second isomorphism below:
\begin{align*}
\Hom_R(\Hom^*_{\cat D}(X,Y),I(\fr)) 
	&\cong \Hom_k(\Hom^*_{\cat D}(X,Y),k) \\
	&\cong \Hom_{\cat D}(Y,\nu(X))\,.
\end{align*}
The first one is by adjunction as $I(\fr)=\Hom_k(R,k)$; see \cite[Lemma~A.2]{Benson/Iyengar/Krause/Pevtsova:2019a}. This is the stated Serre duality on $\tors_\fr(\cat{D})$, for the Krull dimension of $R/\fr$ is $0$ and $R_\fr=R$. 
\end{remark}

\begin{remark}
  Concerning the claim in Remark~\ref{rem:fgc}: Even when $A$ does not
  satisfy the \fgc condition with respect to $R$, the subcategory
  $\tors_{\fr}(\cat D)$ contains the perfect complexes, but it is
  possible that it contains more. To see what is at stake, consider
  the special case that $A$ is self-injective.  Let $M\in \mod A$ be a
  module containing $A$ as a direct summand and satisfying
  $\Ext^i_A(M,M)=0$ for all $i >0$. It is easy to verify that $M$ is
  in $\tors_{\fr}(\cat D)$. However, it is still unknown, and a
  conjecture of Auslander and Reiten~\cite{Auslander/Reiten:1975},
  whether such an $M$ has finite projective dimension, equivalently
  that $M$ is projective.
\end{remark}

\begin{remark}
\label{rem:varyR}
With $k$ and $A$ as before,   let $ R,S$ be connected, finitely generated, $k$-subalgebras of the Hochschild cohomology algebra $\HH^*(A/k)$. Theorem~\ref{thm:gorenstein}, and so also its corollaries, applies  to the action of $R$, and also of $S$, on $\uGProj A$. In reconciling the two, one can replace $S$ by the $k$-subalgebra of $\HH^*(A/k)$ generated by $R$ and $S$ and assume $R\subseteq S$. Then one has an induced map  $\varphi\colon \Proj S\to \Proj R$ defined by the assignment $\fq\mapsto \fq\cap R$.

Given $\fq$ in $\Proj S$, it is clear that there is an inclusion
\[
\gam_{\fq}(\uGProj A) \subseteq \gam_{\fp}(\uGProj A)\quad\text{where $\fp=\fq\cap R$.}
\]
So the version of Theorem~\ref{thm:gorenstein} for the action of $S$ may be seen as a refinement of the one for the action of $R$. Indeed, in the extremal case $R=k$, one has $\fp=0$ for any $\fq$ in $\Spec S$ and $\gam_{\fp}(\uGProj A)=\uGproj A$.

A more interesting situation occurs when $S$ is finite as an $R$-module. Then there are only finitely many primes $\fq$ in $\Proj S$ lying over a given  $\fp\in\Proj R$, and \cite[Corollary~7.10]{Benson/Iyengar/Krause:2012a} yields a direct sum decomposition
\[
\gam_{\fp}(\uGProj A) \cong \bigoplus_{\fq\in\vf^{-1}(\fp)}\gam_{\fq}(\uGProj A)\,.
\]
This decomposition thus  reflects the ramification, in the sense of commutative algebra, of the inclusion $R\subseteq S$.
\end{remark}

\section{Examples}
In this section we describe some examples of Gorenstein algebras. Throughout $k$ will be a field. The \emph{injective dimension} of a finite dimensional $k$-algebra $A$ is the injective dimension of $A$ viewed as a (left) module over itself; we denote it $\injdim A$. This is also the projective dimension of the $A$-module $\Hom_k(A,k)$.

\begin{proposition}
\label{pr:gor-tensor}
Let $\Gamma$ and $\Lambda$ be finite dimensional $k$-algebras. The finite dimensional $k$-algebra $\Gamma\otimes_k\Lambda$ satisfies
\[
\injdim(\Gamma\otimes_k\Lambda)=\injdim \Gamma +\injdim \Lambda\,.
\]
In particular $\Gamma\otimes_k\Lambda$ is Gorenstein if, and only if, both $\Gamma$ and $\Lambda$ are Gorenstein.
\end{proposition}

\begin{proof}
  Set $M:=\Hom_k(\Gamma,k)$ and $N:=\Hom_k(\Lambda,k)$. Let $P$, $Q$
  be minimal projective resolutions of $M$, $N$ respectively. Then
  $P\otimes_kQ$ is a minimal projective resolution of $M\otimes_kN$ by
  the lemma below. It remains to observe that the latter is isomorphic
  to $\Hom_k(\Gamma\otimes_k\Lambda,k)$, as modules over
  $\Gamma\otimes_k\Lambda$.
\end{proof}

\begin{lemma}
Let $P$, $Q$ be minimal projective resolutions of modules
$M$, $N$ over finite dimensional algebras $\Gamma$,
$\Lambda$ respectively. Then $P\otimes_k Q$ is a minimal
projective resolution of the $\Gamma \otimes_k \Lambda$-module
$M \otimes_k N$.
\end{lemma}
\begin{proof}
  With $J(-)$ denoting the Jacobson radical,
  $J(\Gamma)\otimes_k \Lambda + \Gamma \otimes_k J(\Lambda)$ is a
  nilpotent two-sided ideal in $\Gamma \otimes_k \Lambda$, and
  therefore it is contained in $J(\Gamma \otimes_k \Lambda)$. So for
  any projective $\Gamma$-module $U$ and projective $\Lambda$-module
  $V$, we have
\[ \Rad(U)\otimes_k V + U \otimes_k \Rad(V)
\subseteq \Rad(U\otimes_k V) \]
as $\Gamma\otimes_k \Lambda$-modules.
The lemma now follows from the fact that
a projective resolution is minimal if and only if the image of
each differential lands in the radical of the next module.
\end{proof}

Using the result above, one can construct Gorenstein algebras of any
given injective dimension, as long as we find one whose injective
dimension is one. The next example describes such an algebra.  In
particular its class of Gorenstein projective modules is not the same
as the class of Gorenstein injective modules; confer
Remark~\ref{rem:not-self-inj}. Another noteworthy feature of the algebra
is that it is not of finite global dimension.

\begin{example}
\label{ex:GPI}
Let $k$ be a field,  $k[\varepsilon]$ the $k$-algebra of dual numbers, and set
\[
\Lambda:= k[\varepsilon]\otimes_k \Gamma \quad \text{where $\Gamma=\begin{bmatrix} k & k\\0 & k\end{bmatrix}$.}
\]
The $k$-algebra $\Lambda$ has injective dimension one over itself, and can be realised as the path algebra of the quiver
\[
\begin{tikzcd}
   \scriptstyle{1}\arrow[out=150,in=-150,loop,swap,"\varepsilon_1"]\arrow[r,"\alpha"]&
   \scriptstyle{2} \arrow[out=30,in=-30,loop,"\varepsilon_2"]
\end{tikzcd}
\] 
modulo the relations
\[
\varepsilon_1^2=0=\varepsilon_2^2\qquad\textrm{and}\qquad\alpha\varepsilon_1=\varepsilon_2\alpha.
\]

The algebra $\Lambda$ is representation finite and has precisely nine indecomposable modules. There are two simple modules corresponding to the vertices $1$ and $2$. The following diagram shows the Auslander-Reiten quiver. The vertices represent the indecomposables via their composition series. There is a solid arrow $X\to Y$ if there is an irreducible morphism, and a dotted arrow $X\dashrightarrow Y$ when $Y=D\Tr X$.
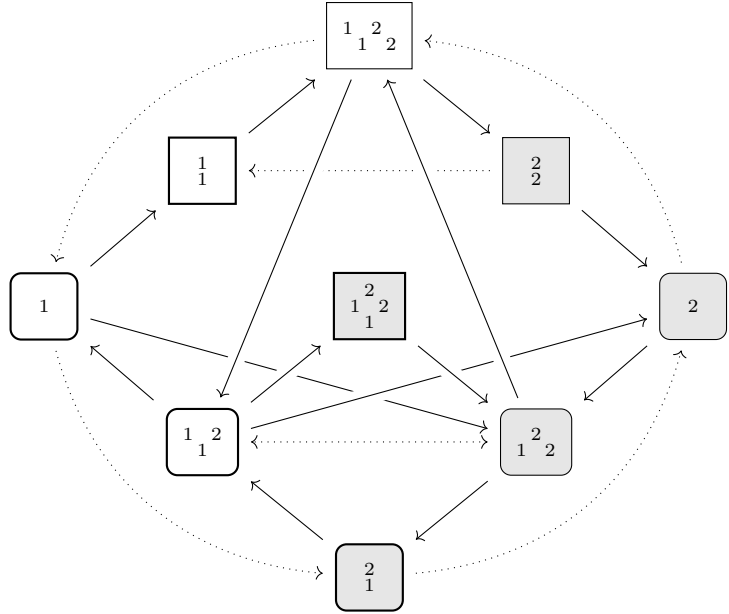
\[
\begin{tikzcd}
   [nodes in empty cells, cells={nodes={minimum height=2.5em}},
   cells={nodes={minimum width=2.5em}}] &&
   \tikz{\node[draw=black,rounded
     corners=0pt]{$\begin{smallmatrix} 1\,\phantom{1}\, 2 \,\phantom{2}\\
       \phantom{1}\,
       1\,\phantom{2}\,2\end{smallmatrix}$}}\arrow[rd]\arrow[lldd,dotted,bend
   right=35]\\
   &\tikz{\node[draw=black,thick,rounded
     corners=0pt]{$\begin{smallmatrix} 1\\
       1\end{smallmatrix}$}}\arrow[ru]&&\tikz{\node[draw=black,rounded
     corners=0pt,fill=black!10]{$\begin{smallmatrix} 2\\
       2\end{smallmatrix}$}}\arrow[rd]\arrow[ll,dotted]\\
   \tikz{\node[draw=black,thick,rounded
     corners]{$\begin{smallmatrix}
       1\end{smallmatrix}$}}\arrow[ru]\arrow[rrrd]\arrow[rrdd,dotted,bend
   right=35]&&
   \tikz{\node[draw=black,thick,rounded corners=0pt,fill=black!10]{$\begin{smallmatrix} 2\\ 1\,\phantom{1}\, 2\\
       1\end{smallmatrix}$}}&& \tikz{\node[draw=black,rounded
     corners,fill=black!10]{$\begin{smallmatrix}
       2\end{smallmatrix}$}} \arrow[ld]\arrow[lluu,dotted,bend
   right=35]\\
   &\tikz{\node[draw=black,thick,rounded corners]{$\begin{smallmatrix} 1\,\phantom{1}\, 2\\
       1\end{smallmatrix}$}}\arrow[lu]\arrow[ru,crossing
   over]\arrow[from=ruuu,crossing
   over]\arrow[rrru,crossing over]&&
   \tikz{\node[draw=black,rounded corners,
     fill=black!10]{$ \begin{smallmatrix} 2\\1\,\phantom{2}\,
       2\end{smallmatrix}$}}\arrow[luuu, crossing over]
   \arrow[ld]\arrow[from=lu,crossing over] \arrow[ll,leftrightarrow,dotted]\\
   &&\tikz{\node[draw=black,thick,rounded corners,
     fill=black!10]{$\begin{smallmatrix} 2\\
       1\end{smallmatrix}$}}\arrow[lu]\arrow[rruu,dotted,bend
   right=35]
 \end{tikzcd}
\]
The Gorenstein projectives have a bold frame, the Gorenstein injectives are shaded, and the modules of finite projective and injective dimension have rectangular shape. A module belongs to all three classes if and only if it is projective and injective; there is a unique indecomposable with this property.

One way to justify these computations is via \cite[Theorem~2]{Ringel/Zhang:2017} due to Ringel and Zhang that sets up a bijection between indecomposable non-projective Gorenstein projective modules over $k[\varepsilon]\otimes_kkQ$ and the indecomposable $kQ$-modules, for any quiver $Q$.
\end{example}

\end{document}